%% file: associative_posets_arxiv.tex
\documentclass[a4paper,11pt]{article}
\usepackage[unicode,hyperfootnotes=false,hidelinks,hyperindex=false]{hyperref}

\input{header-posets}

\usepackage{framed,subcaption}

\hypersetup{
  pdftitle={Associative posets},
  pdfauthor={Joel Kuperman, Alejandro Petrovich, Pedro Sánchez Terraf},
  colorlinks,
  breaklinks,
  linkcolor={blue!35!black},
  citecolor={blue!35!black},
  urlcolor={blue!35!black}  
}
\begin{document}
\title{Definability of band structures on posets}
\author{Joel Kuperman\thanks{
    Universidad Nacional de Córdoba. 
    Facultad de Matemática, Astronomía,  Física y
    Computación. \\
    \indent \phantom{*} Centro de Investigación y Estudios de Matemática (CIEM-FaMAF).
    Córdoba. Argentina.\\
    Supported by Secyt-UNC project 33620230100751CB and Conicet PIP project 11220210100508CO.}
  \and
  Alejandro Petrovich\thanks{
    Departamento de Matemática, Facultad de Ciencias Exactas y
    Naturales, UBA. 
    Departamento de Matemática, Facultad de Ciencias Exactas y
    Naturales, UNLPAM.
  }
  \and
  Pedro Sánchez Terraf\footnotemark[1]
}
\maketitle

\begin{abstract}
  The idempotent semigroups (bands) that give rise to partial orders by defining
  $a \leq b \iff a \cdot b = a$ are the \emph{right-regular} bands (RRB), which are
  axiomatized by $x\cdot y \cdot x = y \cdot x$. In this work we consider the
  class of \emph{associative posets}, which comprises all partial orders
  underlying right-regular bands, and study to what extent the ordering
  determines the possible “compatible” band structures and their canonicity.

  We show that the class of
  associative posets in the signature $\{ \leq \}$ is not first-order
  axiomatizable. We also show that the Axiom of Choice is equivalent over $\ZF$
  to the fact that every tree with finite branches is associative.

  We study the smaller class of “normal” posets (corresponding to right-normal
  bands) and give a structural characterization.
\end{abstract}
\input{intro}
\input{Preliminares}

\input{crarcas}

\input{constructions}
\input{examples}

\input{associative_trees}

\input{conclusion}

\appendix

\input{extra_material}

\input{associative_posets.bbl}

\end{document}

%% file: header-posets.tex
\newcommand{\y}{\wedge}

\usepackage{amsmath}
\usepackage{tikz}
\usetikzlibrary{arrows,shapes}

\newcommand{\R}{\mathbb{R}}
\newcommand{\Q}{\mathbb{Q}}

\usepackage[numbers]{natbib}
\usepackage[utf8]{inputenc}
\usepackage{enumerate}
\usepackage{amsthm}
\usepackage{amssymb} %
\usepackage{amsmath}
\usepackage{verbatim}

\usepackage{stmaryrd} %
\renewcommand{\emptyset}{\varnothing}

\newcommand{\lb}{\langle}
\newcommand{\rb}{\rangle}

\newcommand{\ent}{\Rightarrow}
\newcommand{\tne}{\Leftarrow}

\renewcommand{\phi}{\varphi}

\newcommand{\defi}{\mathrel{\mathop:}=}
\newcommand{\et}{\mathrel{\&}}

\newcommand{\dltr}{\mathrel{\delta}}

\newcommand{\om}{\omega}
\newcommand{\calF}{\mathcal{F}}
\newcommand{\union}{\mathop{\textstyle\bigcup}}
\DeclareUnicodeCharacter{00B7}{\cdot}
\makeatletter
\DeclareRobustCommand\thinbfseries{%
  \not@math@alphabet\bfseries\mathbf
  \fontseries b\selectfont
}
\makeatother

\newcommand{\axiomas}[1]{\mathit{#1}}

\newcommand{\ZF}{\axiomas{ZF}}

\newtheorem{theorem}{Theorem}[section]
\newtheorem{lemma}[theorem]{Lemma}
\newtheorem{prop}[theorem]{Proposition}
\newtheorem{corollary}[theorem]{Corollary}
\newtheorem{claim}[theorem]{Claim}
\theoremstyle{definition}
\newtheorem{question}[theorem]{Question}
\newtheorem{definition}[theorem]{Definition}
\newtheorem{example}[theorem]{Example}
\theoremstyle{remark}
\newtheorem{remark}[theorem]{Remark}

%% file: intro.tex
\section{Introduction}
\label{sec:introduction}

Idempotent semigroups (\emph{bands}) carry a natural
quasiorder structure given by
\begin{equation}\label{eq:quasiorder}
  a \lesssim b \iff a \cdot b = a.
\end{equation}
The associated equivalence relation, $\mathfrak{L} \defi (\lesssim) \cap (\gtrsim)$ (corresponding to Green's equivalence $\mathfrak{l}$ \cite{Green})
is not always a congruence relation, but only for “left-semiregular” ones
\cite{Mistlbacher_1991}. In particular, the variety of \emph{right-regular} bands (RRB),
axiomatized by the equation $x\cdot y \cdot x = y\cdot x$, are characterized by the fact that $\mathfrak{l}$ is
the identity and hence the quasiorder (\ref{eq:quasiorder}) is actually a
partial order.

In this paper, we are interested in studying the class of \emph{associative
posets}, which comprises all partial orders underlying right-regular bands. In
doing so, we will interpret some constructions that give rise to RRBs from an
order-theoretical point of view. This vantage point proves to be useful in a
number of situations, e.g. the characterization of all varieties of bands according
to some particular congruences \cite{Mistlbacher_1991}.

The class of associative posets is ample. It includes all meet-semilattices.
Another source of examples is given by the duals of face posets of a hyperplane
arrangement in $n$-dimensional Euclidean space \cite{brown-face-markov}, and
more generally, any convex subset of faces of such an arrangement (which
includes the case of face posets of convex polyhedra). These examples were
originally presented using \emph{left}-regular bands (“LRB”, satisfying $x\cdot
y\cdot x = x\cdot y$) and using the order defined by
\begin{equation*}\label{eq:lrb-quasiorder}
  a \leq b \iff a \cdot b = b.
\end{equation*}

Contrary to the case of commutative bands, for which the class of underlying
orders (semilattices) is first-order definable, this is not the case for
associative posets (Corollary~\ref{cor:AP-non-first-order}). The search for a
sensible (or structural) characterization of associative posets led us to the
question of definability of such classes.

In the present paper, we prove (Theorem~\ref{teo:arboles}) that the fact that
every tree with finite branches is associative is equivalent over $\ZF$ to the
Axiom of Choice, thereby showing that there is no canonical assignment of a band structure to each
associative poset.

This discussion of definability permeates our whole work, and many
questions remain open; some of these are gathered in Section~\ref{sec:conclusion}.
After setting up some preliminaries in the next section,
we restrict ourselves in Section~\ref{sec:normal-posets} to a subfamily
of associative posets for which a neat structural characterization is
available.

%% file: Preliminares.tex
\section{Preliminaries}
\label{sec:preliminaries}

For any poset $(P,\leq)$, we  say that a binary operation
$\cdot$ on $P$ is \emph{admissible for} $(P,\leq)$ whenever $x\leq y$ if and only if
\begin{equation}
  \label{eq:order_of_RRB}
  x\cdot y=x.
\end{equation}
for all $x,y$ in $P$. In the case of an associative poset, we can moreover
choose this to be an RRB operation and then we use the phrase \emph{right posemigroup} when
referring to the expanded structure $(P,\leq, \cdot)$; this stems from the fact
that the product is order-preserving on the right only (see
Lemma~\ref{lem:lower-bounds}.\ref{item:monot-prod}).
Conversely,
given an RRB $(P,\cdot)$, we say that the partial order $x \leq y$
given by (\ref{eq:order_of_RRB}) is the \emph{underlying order} of
this RRB.

We start by introducing some useful elementary properties of right posemigroups.
\begin{lemma}\label{lem:aba-ba}
 In every right posemigroup,
  \begin{enumerate}
  \item \label{item:1} $a · b \leq b$; in particular, if $b$ is
    minimal, $a · b = b$.
  \item \label{item:aba-ba} $a · b · a = b · a$.
  \item \label{item:3} $a \leq b \implies b· a =a$.
  \end{enumerate}
\end{lemma}
\begin{proof}
  \begin{enumerate}
  \item $(a · b ) · b = a · (b · b) = a · b$.
  \item From Item~\ref{item:1} we know that
    $a ·(b ·a) \leq b ·a$ and $b·a = b ·(a ·b ·a)\leq a ·b ·a$.  By
    antisymmetry we obtain $a ·b ·a = b ·a$.
  \item $a = a · a = (a · b) · a = b ·a$, by
    Item~\ref{item:aba-ba}. \qedhere
  \end{enumerate}
\end{proof}

\begin{definition} 
Let $P$ be a poset and $D$ a subset of $P$. We say that $D$ is a \emph{decreasing subset} of $P$ if $\forall a, b\in P, a\in D \et b\leq a \Rightarrow b \in D$.
\end{definition}

\begin{definition}
Let $A$ be a right posemigroup and $B$ a subset of $A$. We say that $B$ is a substructure of $A$ if $B$ is closed under the semigroup operation. 
\end{definition}

\begin{corollary}\label{cor:decr-subalg}
  Every decreasing subset of a right posemigroup is closed under the semigroup operation, and hence it is the universe of a substructure.
\end{corollary}
  
\begin{lemma}\label{lem:lower-bounds}
  For every right posemigroup, 
  \begin{enumerate}
  \item \label{item:monot-prod}$a\leq b$ implies $ a · x \leq b · x$.
  \item \label{item:iguales-inf} $c\leq x,y $ implies $c\leq x·y,
    y·x$. Hence if $x·y= y·x$, it must be  the infimum of $\{x,y \}$.
  \end{enumerate}
\end{lemma}
\begin{proof}
  \begin{enumerate}
  \item Using Lemma~\ref{lem:aba-ba}(\ref{item:aba-ba}),
    $a·\underline{x·b·x} = \underline{a·b}·x = a ·x$.
  \item $c= \underline{c}·x = c· y·x$. Hence $c\leq
    y·x$. Symmetrically, $c\leq x·y$. \qedhere
  \end{enumerate}
\end{proof}
In the following, we use $x{\downarrow}$ to denote $\{ a\in P \mid a\leq x \}$.
\begin{lemma}\label{lema:isomorfos}
  Let $P$ be a right posemigroup. If $x\cdot y=y$ then $(y\cdot
  x){\downarrow}$ is isomorphic to $y{\downarrow}$.
\end{lemma}
\begin{proof}[Proof]
  We prove that the function $f:(y\cdot x){\downarrow}\to
  y{\downarrow}$ defined by $f(a)\defi a\cdot y$ is an isomorphism with inverse
  $b \mapsto b \cdot x$. The injectivity
  of $f$ follows from the fact that for all
  $a\leq y\cdot x$,  we have $a=a\cdot(y\cdot x)=(a\cdot y)\cdot
  x$. Now, if $b\leq y$, take $a=b\cdot x$. We know $a\leq y\cdot x$ by
  Lemma~\ref{lem:lower-bounds}(\ref{item:monot-prod}). Now, $f(a)=(b\cdot x)\cdot
  y=b\cdot (x\cdot y)=b\cdot y=b$. So $f$ is surjective and hence a bijection. To see that $f$ preserves the right posemigroup operation, take $a,b\in (y\cdot x){\downarrow}$. We have that $f(a\cdot b)= (a\cdot b) \cdot y = a \cdot (b\cdot y) = a\cdot (y\cdot b\cdot y) = (a\cdot y)\cdot (b\cdot y) =f(a)\cdot f(b)$. 
\end{proof}
\begin{corollary}\label{cor:xy-yx}
  Let $P$ be a right posemigroup. For every $x,y\in P$, $(x\cdot
  y){\downarrow}$ is isomorphic to $(y\cdot x){\downarrow}$.
\end{corollary}
\begin{proof}
  We have $(x\cdot y)\cdot (y\cdot x)=x\cdot (y\cdot y)\cdot x=x\cdot
  y\cdot x=y\cdot x$. Symmetrically, $(y\cdot x)\cdot (x\cdot
  y)=x\cdot y$. By Lemma~\ref{lema:isomorfos} we have $(x\cdot
  y){\downarrow}\approx (y\cdot x){\downarrow}$.
\end{proof}
\begin{corollary}\label{cor:minimal}
  Let $P$ be a right posemigroup and $m\in P$ be minimal. For every
  $x\in P$, $m\cdot x\leq x$ is minimal.
\end{corollary}

%% file: crarcas.tex
\section{Normal posets}
\label{sec:normal-posets}

It is well known \cite{Ger1} that there are four proper subvarieties of RRBs,
each of which can be axiomatized by one extra identity besides the band
axioms. These are:
\begin{itemize}

 \item $x = y$. The trivial variety.

\item $x \cdot y = y \cdot x$. The class of partial orders underlying the
  subvariety of commutative bands is exactly the class of (meet-)semilattices.

\item $x \cdot y= y$. These are the “right-zero” bands, and the class of
  partial orders underlying this one is the class of antichains.

\item $x \cdot y \cdot z = y\cdot x \cdot z$. These are a superset of right-zero
  bands called “right-normal” (RNB). 
\end{itemize}

Our goal for this section is to give a characterization of \emph{normal posets},
which are the orders underlying right-normal bands. Normal posets are
always \emph{relative meet-semilattices} (i.e., posets in which every principal ideal is a
meet-semilattice) but not conversely (even assuming associativity; see
Example~\ref{exm:tulip}). In fact, it is not hard to show that an associative
relative meet-semilattice is a normal poset if and only if it admits an
operation which acts as the meet operation in every principal ideal \cite[Lemma~10]{pogroupoids}.

First, we require an auxiliary result. Recall that in every band $B$, Green's
equivalence $\mathfrak{R}$ is characterized by $x\cdot y = y \et y\cdot x =
x$. Also, the equivalence $\mathfrak{D}$ given by the composition of $\mathfrak{R}$ and $\mathfrak{L}$, is
the least congruence such that $B/\mathfrak{D}$ is a semilattice \cite{semilattice_cong}.
Since on every right regular band, Green’s equivalence $\mathfrak{L}$
is trivial, we have $\mathfrak{D}$ = $\mathfrak{R}$ on right regular bands.
From this, we can obtain the following Lemma:
\begin{lemma}
Let $B$ be an RRB. Then $\mathfrak{D}=\{(x,y)\in B^2: x\cdot y=y \et y\cdot
  x=x\}$ is the least congruence such that $B/\mathfrak{D}$ is a semilattice. 
\end{lemma}

In view of this result, for any RRB $B$, we shall call $\mathfrak{D}$ the \emph{semilattice congruence} of
$B$ and $B/\mathfrak{D}$ the \emph{quotient semilattice} of $B$.

\begin{theorem}~\label{th:normal}
  Let $P$ be a poset. The following are equivalent:
  \begin{enumerate}
  \item\label{item:normal}
    $P$ is normal;
  \item\label{item:charact-normal}
    There exist a meet-semilattice $S$ and an order homomorphism $f:P\rightarrow
    S$ which satisfies that $f_m := f|_{m{\downarrow}}:m{\downarrow}\rightarrow
    f(m){\downarrow}$ is an isomorphism between $m{\downarrow}$ and
    $f(m){\downarrow}$ for every $m\in P$.
  \end{enumerate}
\end{theorem}
It is immediate that in \ref{item:charact-normal}, it is enough to verify the
condition for each $m$ in some cofinal $M\subseteq P$.
\begin{proof}
  ($\ent$) Fix an admissible RNB operation $\cdot$ for
  $P$.  Let $\mathfrak{D}$ be the semilattice congruence of $(P,\cdot)$ and
  $f:P\rightarrow P/\mathfrak{D}$ the canonical projection. First note that since each
  initial segment $p{\downarrow}$ is decreasing, it is also a substructure, so
  the restriction of any homomorphism to it is a homomorphism. We also have that
  RRB isomorphisms are isomorphisms of the underlying orders. Therefore we only need to check that
  $f|_{p{\downarrow}} = f_p$ is bijective for every $p \in P$. Let $p\in P$; since
  $(P,\cdot)$ is an RNB, if $x,y\leq p$ and $f(x)=f(y)$, then
  \[
    x=x\cdot p=y\cdot x\cdot p=x\cdot y\cdot p=y\cdot p=y
  \]
  since $x\mathrel{\mathfrak{D}} y$ and hence $f_p$ is injective. Also, if $x/\mathfrak{D}\leq
  p/\mathfrak{D}$, then $x\cdot p/\mathfrak{D}= x/\mathfrak{D}$, so $f_p(x\cdot p)=x/\mathfrak{D}$. Thus
  $f_p$ is an order isomorphism for every $p\in P$.

  ($\tne$) Assume that $S$
  and $f:P\rightarrow S$ satisfy \ref{item:charact-normal}. Consider the
  antichain order $A\defi(P,=)$ on $P$. Let $B\subseteq 
  S\times A$ given by 
  \[
    B:=\{\lb x,m\rb \mid x\leq f(m) \}=\{\lb f(m),m\rb \mid m\in
    P\}{\downarrow}.
  \]
  $B$ is a substructure of the direct product RNB structure on $S\times A$. Let
  $h:B\rightarrow P$ given by $h(\lb x,m\rb)=(f_m)^{-1}(x)$. Note that $h$ is
  surjective.
  \begin{claim}\label{claim:ker-is-congruence}
    $\delta \defi \ker h = \{ \lb x,y\rb \mid h(x) = h(y) \}$ is a congruence over $B$.
  \end{claim}
  Indeed, if $\lb x,m \rb \dltr \lb x',m'\rb$ and $\lb y,n\rb\dltr \lb y',n'\rb$, then
  $(f_m)^{-1}(x)=(f_{m'})^{-1}(x')$, and so
  $x=f((f_m)^{-1}(x))=f((f_{m'})^{-1}(x'))=x'$. Analogously, using
  that $(f_n)^{-1}(y)=(f_{n'})^{-1}(y')$ we see that $y=y'$.  Now we
  have
  \[
    h(\lb x,m\rb\cdot \lb y,n \rb)=h(\lb x\wedge y,n\rb)=(f_n)^{-1}(x\wedge y)=:p,
  \]
  while
  \[
    h(\lb x',m'\rb\cdot \lb y',n'\rb)=(f_{n'})^{-1}(x'\wedge
    y')=(f_{n'})^{-1}(x\wedge y)=:q.
  \]
  As $q\leq
  (f_{n'})^{-1}(y)=(f_n)^{-1}(y)\leq n$, and $f(q)=x\wedge y=f(p)$, we
  must have $p=q$ as both $p$ and $q$ belong to $n{\downarrow}$.
  \begin{claim}
    $\phi:B/\ker h\rightarrow P$ given by $\phi([\lb x,m\rb])=h(\lb x,m\rb )$ is an order isomorphism.
  \end{claim}
 The map $\phi$ is clearly bijective so we only have to check that it is
  order preserving and that is has an order preserving inverse. Given $[\lb x,m\rb],[\lb y,m'\rb]\in B/\ker h$ such that
  $[\lb x,m\rb]\cdot[\lb y,m'\rb]=[\lb x,m\rb]$, we have $[\lb x,m\rb]=[\lb x\wedge y,m'\rb]$,
  so
  \[
    \phi([\lb x,m\rb])=\phi [\lb x\wedge y, m'\rb]=(f_{m'})^{-1}(x\wedge y)\leq
    (f_{m'})^{-1}(y)=\phi ([\lb y,m'\rb]).
  \]
  We also have that $\phi^{-1}(p)=[\lb f(p),p\rb]$. If $p\leq q$, then
  \[
    [\lb f(p),p\rb]\cdot [\lb f(q),q\rb]=[\lb f(p),q\rb]=[\lb f(p),p\rb]=\phi^{-1}(p)
  \]
  Then $\phi$ is a poset
  isomorphism. Therefore, the RNB structure of
  $B/\ker h$ is admissible for $P$.
\end{proof}

\begin{example}\label{ex:norm}
Applying the previous theorem we can see that the poset depicted in Figure~\ref{fig:normal} is normal.

\begin{figure}[h]
\begin{center}

\begin{tikzpicture}
   [>=latex, thick,
          % label distance=0ex,
          nodo/.style={thick,minimum size=0cm,inner sep=0cm}]
       
        \node (7) at (1.75,1) [color=black!100, fill=black!0, nodo]  {$\bullet$};
         \node (8) at (2.25,1) [color=black!100, fill=black!0, nodo]  {$\bullet$};
        \node (9) at (1.75,.5) [color=black!100, fill=black!0, nodo] {$\bullet$};
        \node (10) at (2.25,.5) [color=black!100, fill=black!0, nodo] {$\bullet$};
        \node (11) at (2,0) [color=black!100, fill=black!0, nodo] {$\bullet$};

        \draw [-] (9) edge  (7);
        \draw [-] (10) edge  (7);
        \draw [-] (9) edge  (8);
        \draw [-] (10) edge  (8);
        \draw [-] (11) edge  (9);
        \draw [-] (10) edge  (11);
        
\end{tikzpicture}

\caption{A normal poset}\label{fig:normal}
\end{center}
\end{figure}
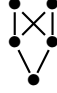
\begin{proof}
Consider the equivalence relation over the previous poset in which the only non-trivial pair consists of the two maximal elements of the poset. The quotient of this relation is a meet-semilattice and the canonical projection satisfies the conditions of the previous theorem.
\end{proof}
\end{example}

\begin{example} \label{ex:mult_norm}
  For the leftmost poset from Figure~\ref{fig:multiple-RNB-struct}, different
  homomorphisms can be chosen. Each one gives rise to a different compatible
  right-normal band operation.
  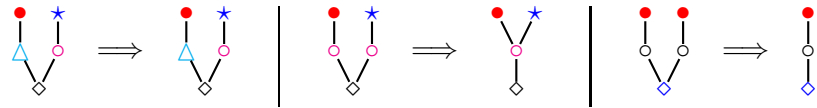
\begin{figure}[h!]
    \begin{center}
      \begin{tabular}{lllc|ccccc|crrr}
        \begin{tikzpicture}
          [>=latex, thick,
            % label distance=0ex,
            nodo/.style={thick,minimum size=0cm,inner sep=0cm}]
          
          \node (7) at (1.75,1) [color=red!100, fill=red!0, nodo]  {$\bullet$};
          \node (8) at (2.25,1) [color=blue!100, fill=black!0, nodo]  {$\star$};
          \node (9) at (1.75,.5)  [color=cyan!100, fill=green!0, nodo] {$\vartriangle$};
          \node (10) at (2.25,.5) [color=magenta!100, fill=cyan!0, nodo] {$\circ$};
          \node (11) at (2,0)  [color=black!100, fill=blue!0, nodo] {$\diamond$};

          \draw [-] (9) edge  (7);
          \draw [-] (10) edge  (8);
          \draw [-] (11) edge  (9);
          \draw [-] (10) edge  (11);
          
        \end{tikzpicture} & 
        \raisebox{3ex}{$\Longrightarrow$} & \begin{tikzpicture}
          [>=latex, thick,
            % label distance=0ex,
            nodo/.style={thick,minimum size=0cm,inner sep=0cm}]
          
          \node (7) at (1.75,1) [color=red!100, fill=red!0, nodo]  {$\bullet$};
          \node (8) at (2.25,1) [color=blue!100, fill=black!0, nodo]  {$\star$};
          \node (9) at (1.75,.5) [color=cyan!100, fill=green!0, nodo] {$\vartriangle$};
          \node (10) at (2.25,.5) [color=magenta!100, fill=cyan!0, nodo] {$\circ$};
          \node (11) at (2,0) [color=black!100, fill=blue!0, nodo] {$\diamond$};

          \draw [-] (9) edge  (7);
          \draw [-] (10) edge  (8);
          \draw [-] (11) edge  (9);
          \draw [-] (10) edge  (11);
          
        \end{tikzpicture} & & &
        \begin{tikzpicture}
          [>=latex, thick,
            % label distance=0ex,
            nodo/.style={thick,minimum size=0cm,inner sep=0cm}]
          
          \node (7) at (1.75,1) [color=red!100, fill=red!0, nodo]  {$\bullet$};
          \node (8) at (2.25,1) [color=blue!100, fill=black!0, nodo]  {$\star$};
          \node (9) at (1.75,.5) [color=magenta!100, fill=green!0, nodo] {$\circ$};
          \node (10) at (2.25,.5) [color=magenta!100, fill=cyan!0, nodo] {$\circ$};
          \node (11) at (2,0) [color=black!100, fill=blue!0, nodo] {$\diamond$};

          \draw [-] (9) edge  (7);
          \draw [-] (10) edge  (8);
          \draw [-] (11) edge  (9);
          \draw [-] (10) edge  (11);
          
        \end{tikzpicture}&  \raisebox{3ex}{$\Longrightarrow$}  &\begin{tikzpicture}
          [>=latex, thick,
            % label distance=0ex,
            nodo/.style={thick,minimum size=0cm,inner sep=0cm}]
          \node (7) at (1.75,1) [color=red!100, fill=red!0, nodo]  {$\bullet$};
          \node (8) at (2.25,1) [color=blue!100, fill=black!0, nodo]  {$\star$};
          \node (9) at (2,.5) [color=magenta!100, fill=green!0, nodo] {$\circ$};
          \node (11) at (2,0) [color=black!100, fill=blue!0, nodo] {$\diamond$};

          \draw [-] (9) edge  (7);
          \draw [-] (9) edge  (8);
          \draw [-] (11) edge  (9);
        \end{tikzpicture} & & &
        \begin{tikzpicture}
          [>=latex, thick,
            % label distance=0ex,
            nodo/.style={thick,minimum size=0cm,inner sep=0cm}]
          
          \node (7) at (1.75,1) [color=red!100, fill=red!0, nodo]  {$\bullet$};
          \node (8) at (2.25,1) [color=red!100, fill=black!0, nodo]  {$\bullet$};
          \node (9) at (1.75,.5) [color=black!100, fill=green!0, nodo] {$\circ$};
          \node (10) at (2.25,.5) [color=black!100, fill=cyan!0, nodo] {$\circ$};
          \node (11) at (2,0) [color=blue!100, fill=blue!0, nodo] {$\diamond$};

          \draw [-] (9) edge  (7);
          \draw [-] (10) edge  (8);
          \draw [-] (11) edge  (9);
          \draw [-] (10) edge  (11);
          
        \end{tikzpicture} &  \raisebox{3ex}{$\Longrightarrow$} &  \begin{tikzpicture}
          [>=latex, thick,
            % label distance=0ex,
            nodo/.style={thick,minimum size=0cm,inner sep=0cm}]
          \node (7) at (2,1) [color=red!100, fill=red!0, nodo]  {$\bullet$};
          
          \node (9) at (2,.5) [color=black!100, fill=green!0, nodo] {$\circ$};
          \node (11) at (2,0) [color=blue!100, fill=blue!0, nodo] {$\diamond$};

          \draw [-] (9) edge  (7);

          \draw [-] (11) edge  (9);
        \end{tikzpicture} 

      \end{tabular}
    \end{center}
    \caption{Multiple RNB structures over a normal poset. We define three different equivalence relations. Two elements are related if and only if their nodes are of the same color and shape. One can easily check that the quotient and the canonical projection satisfy the hypotheses of the previous theorem in each case. }\label{fig:multiple-RNB-struct}
  \end{figure}

\end{example}
Note that the quotient semilattice of the RNB obtained in the last part of the
proof is isomorphic to the semilattice $S$ with which we started. Also the
canonical projection of the semilattice congruence is the homomorphism $f$. It
is not hard to see that in the case of right-normal bands, one can completely determine the
product using two equationally definable binary relations: the
underlying partial order and the semilattice congruence. This tells us,
according to Beth's Theorem, that there is a first order formula $\phi(x,y,z)$
in the language $\{ \leq, \mathfrak{D} \}$ which is equivalent to the formula $x\cdot
y=z$. In fact, Theorem~\ref{th:normal} provides information for defining it. We
simply need a formula which roughly says ``$z$ is the representative of the
class $x/\mathfrak{D} \wedge y/\mathfrak{D}$ which is below $y$''. For this, we define $\psi(x,y)$ which states ``$x/\mathfrak{D} \leq y/\mathfrak{D}$'', as $\psi(x,y) := \exists c,\; c\leq y \et x \mathrel{\mathfrak{D}} c$. Then we set
\[
  \phi(x,y,z) := z\leq y \et \psi(z,x) \et \forall d,\; (\psi(d,x)\et
  \psi(d,y))\rightarrow \psi(d,z).
\]
 
\begin{example}
The non-normal poset in Figure~\ref{fig:non-normal} admits two non isomorphic right-regular band operations. The semilattice congruence is the same in both right-regular bands.

\begin{figure}[h]
\begin{center}
\begin{tikzpicture}
   [>=latex, thick,
          % label distance=0ex,
          nodo/.style={thick,minimum size=0cm,inner sep=0cm}]

         \node (6) at (-1,0) [color=black!100, fill=black!0, nodo] {$\bullet$};
        \node (7) at (2,1) [color=black!100, fill=black!0, nodo] {$\bullet$};
         \node (8) at (1,0) [color=black!100, fill=black!0, nodo] {$\bullet$};
        \node (9) at (3,0) [color=black!100, fill=black!0, nodo] {$\bullet$};
        \node (10) at (3,-1) [color=black!100, fill=black!0, nodo]  {$\bullet$};

        \draw [-] (8) edge  (7);
        \draw [-] (10) edge  (9);
        \draw [-] (9) edge  (7);        
\end{tikzpicture}
\caption{An associative poset with a single semilattice congruence.}\label{fig:non-normal}
\end{center}
\end{figure}
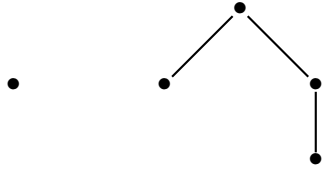
It's easy to see that the previous poset is not a normal poset as it is not even a relative meet-semilattice. Associativity stems from Remark~\ref{obs:foresta}.

\end{example}

%%% Local Variables: 
%%% mode: latex
%%% TeX-master: "associative_posets"
%%% End:

%% file: constructions.tex
\section{Constructions}
\label{sec:constructions}

In this section we discuss closure under certain model-theoretic operations of
the class of associative posets. These are obviously correlated to algebraic
constructions on the RRB side.

Given two posets $(P,\leq^P)$ and $(Q,\leq^Q)$ with $P$ disjoint from
$Q$, we define their \emph{disjoint union} $\bigl(P\sqcup
Q,({\leq^P})\sqcup({\leq^Q})\bigr)$, in which the copies of both posets are
unrelated.
We extend this concept in the obvious way
to define the disjoint union of a family of posets, $\bigsqcup_{i\in
  I}{P}_i$. We also define their \emph{ordered sum} $P+Q$ as the poset
\[
  \bigl(P\sqcup Q,({\leq^P})\sqcup({\leq^Q})\sqcup (P\times Q)\bigr),
\]
and analogously for a family of posets, $\sum_{i\in I} P_i$, for any linearly
ordered index set $I$.

\begin{lemma}\label{lem:sum-ord}
  Let $\{P_i : i\in I\}$ be a family of associative posets. Then
    $\sum_{i\in I}P_i$ is associative.
\end{lemma}
\begin{proof}
  First, we prove that the sum of two associative posets $P$ and $Q$ is
  associative.  We define a right posemigroup operation over $P+Q$ in the
  following manner:
  \[
    x\cdot y=\begin{cases} \min\{x,y\} & x,y \text{ comparable} \\
    x\cdot^Q y & \text{$x,y\in Q$}\\ x\cdot^P y
    & \text{$x,y\in P$}\\
    \end{cases}
  \]
  A straightforward case analysis shows  associativity.
  By induction, a
  sum of a finite family of associative posets is associative.
  
  Now let $\{P_i: i\in I \}$ be a family of associative posets and let
  $\cdot_i$ denote an admissible RRB structure over $P_i$. Consider
  the following product over  $\sum_{i\in I}P_i$:
  \[
    x\cdot y=
    \begin{cases}
      \min\{x,y\} & \text{$x$ and $y$ comparable} \\
      x\cdot_i y & \text{$x,y\in P_i$}
    \end{cases}
  \]
  Let $x,y,z$ be arbitrary elements in
  $\bigsqcup_{i\in I}P_i$ and  $i,j,k\in I$ be such that $x\in P_i$,
  $y\in P_j$, $z\in P_k$. Now we have, by associativity of
  $P_i+P_j+P_k$ and the fact that this finite sum is a substructure of
  $\sum_{i\in I}P_i$, that $(x\cdot y)\cdot z=x\cdot (y\cdot z)$.
\end{proof}

\begin{lemma}\label{lem:union-disj}
Let $P$ be an associative (resp., normal) poset, $I$ a set and $P_i$ an isomorphic
copy of $P$ for each $i\in I$. Then $\bigsqcup_{i\in I} P_i$ is
associative (normal).
\end{lemma}
\begin{proof}
  Observe that $\bigsqcup_{i\in I} P_i$ is the underlying order of the
  right-regular (resp., normal) given by $(P,\cdot) \times (I,\pi_2)$, where $\cdot$ is an
  admissible right-regular (normal) band structure for $P$ and $\pi_2(x,y)=y$ for any $x,y\in
  I$.
\end{proof}
Similar arguments also yield:
\begin{lemma}\label{lem:sum-ord2}
  Let $\{Q_i:i\in I\}$ be a family of associative posets and let $P$
    an associative poset with top element $1$. Let us define
    $R_i=P_i+Q_i$, with $P_i$ an isomorphic copy of $P$, then
    $\bigsqcup_{i\in I} R_i$ is associative.
\end{lemma}

\begin{theorem}\label{th:normal-posets-not-first-order}
  The class of normal posets is not axiomatizable by first-order
  sentences.
\end{theorem}
\begin{proof}
  Consider the disjoint union $P=\R\sqcup\R$, which is normal by
  Lemma \ref{lem:union-disj}. It can be shown by using Ehrenfeucht–Fraïssé
  games that $\R\sqcup \Q$ is elementarily equivalent to $P$. But by
  Lemma \ref{lema:isomorfos} this poset is even not associative. Therefore,
  the class of normal posets is not closed under elementary
  equivalence and therefore is not a first-order class.
\end{proof}
  
\begin{corollary}\label{cor:AP-non-first-order}
  The class of associative posets is not axiomatizable by first-order
  sentences.\qed
\end{corollary}

%% file: examples.tex
\section{Examples}

Every poset $P$ admits a binary operation $\cdot$ in the sense of
(\ref{eq:order_of_RRB}). This operation can be chosen to be commutative if $P$
is not the two-element antichain. As it is well-known, meet-semilattices
are characterized by the fact that
$\cdot$ can be (uniquely) chosen to be commutative and associative; therefore
every meet-semilattice is immediately a normal poset and hence associative.

The following is the smallest example of a non-associative poset.
\begin{example}[The hummingbird]
  The poset in Figure~\ref{fig:picaflor} is not associative. Assume by
  way of contradiction that it admits an RRB structure $·$. By
  Lemmas~\ref{lem:lower-bounds}(\ref{item:iguales-inf}) and~\ref{lem:aba-ba}(\ref{item:1}), $b · x = x$
  and $x · b = b$; but this contradicts
  Corollary~\ref{cor:xy-yx}, since $x{\downarrow}$ and $b{\downarrow}$ are
  obviously not isomorphic.
  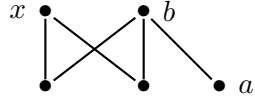
\begin{figure}[h]
    \begin{center}
      \begin{tikzpicture}
        [>=latex, thick,
          nodo/.style={thick,minimum size=0cm,inner sep=0cm}]
        
        \node (a) at (2.3,0) [label=right:$a$] [nodo] {$\bullet$};
        \node (x) at (0,1) [label=left:$x$]   [nodo] {$\bullet$};
        \node (0) at (0,0)  [nodo] {$\bullet$};
        \node (b) at (1.3,1)  [label=right:$b$]  [nodo] {$\bullet$};
        \node (1) at (1.3,0)  [nodo] {$\bullet$};
        
        \draw [-] (x) edge  (0);
        \draw [-] (x) edge  (1);
        \draw [-] (b) edge  (0);
        \draw [-] (b) edge  (1);
        \draw [-] (b) edge  (a);
      \end{tikzpicture}
      \caption{The hummingbird.} \label{fig:picaflor}
    \end{center}
  \end{figure}
\end{example}

\begin{example}[The 3-crown]
  The poset depicted in Figure~\ref{fig:crown} is also not associative, but for
  totally different reasons (see Appendix~\ref{sec:additional-proofs} for a
  proof). It is noteworthy that analogous $(2n)$-crowns of even width are all
  associative.
  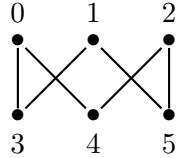
\begin{figure}[h]
    \begin{center}
      \begin{tikzpicture}
        [>=latex, thick,
          nodo/.style={thick,minimum size=0cm,inner sep=0cm}]
        
        \node (0) at (-1,1) [label=above:$0$] [nodo] {$\bullet$};
        \node (1) at (0,1) [label=above:$1$]   [nodo] {$\bullet$};
        \node (2) at (1,1)  [label=above:$2$]  [nodo] {$\bullet$};
        \node (3) at (-1,0)  [label=below:$3$]  [nodo] {$\bullet$};
        \node (4) at (0,0)  [label=below:$4$] [nodo] {$\bullet$};
        \node (5) at (1,0)  [label=below:$5$] [nodo] {$\bullet$};
        
        \draw [-] (0) edge  (3);
        \draw [-] (0) edge  (4);
        \draw [-] (1) edge  (3);
        \draw [-] (1) edge  (5);
        \draw [-] (2) edge  (4);
        \draw [-] (2) edge  (5);
      \end{tikzpicture}
      \caption{The crown poset.\label{fig:crown}} 
    \end{center}
  \end{figure}  
\end{example}

\begin{example}[The puppy]
The poset depicted in Figure~\ref{fig:puppy} is an example of an associative poset in which there is no RRB operation for which $x\cdot y = x \wedge y$ holds for every pair of elements $x, y$ such that $x\wedge y$ exists. In the only admissible RRB operation for this poset, $a\cdot b = b$ and $b\cdot a = a$.
\begin{figure} [h]
  \begin{center}
      \begin{tikzpicture}
        [>=latex, thick,
          nodo/.style={thick,minimum size=0cm,inner sep=0cm}]
        
        \node (a) at (0.5,0.5) [label=left:$a$] [nodo] {$\bullet$};
        \node (x) at (0,1)   [nodo] {$\bullet$};
        \node (0) at (0,0)  [nodo] {$\bullet$};
        \node (y) at (2,1)    [nodo] {$\bullet$};
        \node (1) at (2,0)  [nodo] {$\bullet$};
        \node (b) at (1.5,0.5) [label=right:$b$] [nodo] {$\bullet$};
        \node (c) at (1,0)[label=right:$c$] [nodo] {$\bullet$};
        
        \draw [-] (y) edge  (b);
        \draw [-] (x) edge  (0);
        \draw [-] (x) edge  (1);
        \draw [-] (y) edge  (0);
        \draw [-] (y) edge  (1);
        \draw [-] (x) edge  (a);
        \draw [-] (c) edge  (a);
        \draw [-] (c) edge  (b);
        
      \end{tikzpicture}
      \caption{The puppy.} \label{fig:puppy}
    \end{center}
  \end{figure}
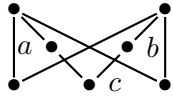
\end{example}

\begin{example}[The tulip]\label{exm:tulip}
  The poset depicted in Figure~\ref{fig:tulip} is an example of a non-normal
  associative relative meet-semilattice (every initial segment is a
  meet-semilattice).
  The situation is analogous to the previous example; we must have $a\cdot b =
  b$ and $b\cdot a = a$.
  \begin{figure} [h]
    \begin{center}
      \begin{tikzpicture}
        [>=latex, thick,
          nodo/.style={thick,minimum size=0cm,inner sep=0cm}]
        
        \node (x) at (-.5,2) [nodo] {$\bullet$};
        \node (y) at (.5,2)  [nodo] {$\bullet$};
        \node (z) at (1.5,2)  [nodo] {$\bullet$};
        \node (a) at (-.5,1)  [label=left:$a$]  [nodo] {$\bullet$};
        \node (b) at (1.5,1) [label=right:$b$] [nodo] {$\bullet$};
        \node (c) at (0,1) [nodo] {$\bullet$};
        \node (d) at (1,1) [nodo] {$\bullet$};
        \node (0) at (0.5,0) [label=right:$0$] [nodo] {$\bullet$};
        
        \draw [-] (x) edge  (a);
        \draw [-] (x) edge  (c);
        \draw [-] (x) edge  (d);
        \draw [-] (y) edge  (a);
        \draw [-] (y) edge  (b);
        \draw [-] (z) edge  (c);
        \draw [-] (z) edge  (d);
        \draw [-] (z) edge  (b);
        \draw [-] (0) edge  (a);
        \draw [-] (0) edge  (c);
        \draw [-] (0) edge  (d);
        \draw [-] (0) edge  (b);
      \end{tikzpicture}
      \caption{The tulip.} \label{fig:tulip}
    \end{center}
  \end{figure}
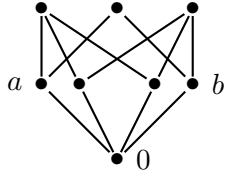
\end{example}

We will end this section by showing that the order dual of $\mathbb{N}\times\mathbb{N}$ admits
only one RRB structure.
It is straightforward to check that the only RRB operation
admissible for a chain is the infimum. The next lemma extends this
observation.
\begin{lemma}
  Let $C_1$ and $C_2$ be chains, and at least one of them has no
  minimum. Then the only admissible RRB on the direct
  product $C_1\times C_2$ is given by the infimum.
\end{lemma}
\begin{proof}
  Let $\leq^i$ denote the order in $C_i$, and $\leq$ the direct
  product order.
  Now assume, by way of
  contradiction, that there are $a,b\in L$ such that 
  $a\cdot b\neq a\y b$. Without loss of generality, assume that 
  $a\cdot b = b$ and $b\cdot a = a$.

  Assume that $c<a$ and $c\nleq b$ for a fixed $c$. Then we have
  \[
  (c\cdot b)\cdot a = c\cdot (b\cdot a) = c\cdot a = c.
  \]
  Hence we deduce that $c\cdot b<b$, $c\cdot b\neq c$, and $c\cdot b\nleq a$.
  From the last one we have $c\y b < c\cdot b$.
  
  Again without loss of generality, we may assume $a =\lb a_1,a_2\rb$,
  $b =\lb b_1,b_2\rb$ with $a_1<^1 b_1$ and 
  $b_2 <^2  a_2$. Also let  $a_1'<^1 a_1$,  
  $c\defi \lb a_1',a_2\rb$ and $s_1,s_2\in\om$ such
  that   $c\cdot b=\lb s_1,s_2\rb$. We have $c\y b = \lb a_1',b_2\rb$.
  Since  $c<a$ and $c\nleq b$, the calculations of the previous
  paragraph apply.
  From  $c\cdot b<b$ and $c\cdot b\nleq a$ we have
  \begin{equation}\label{eq:1}
    s_2 \leq^2 b_2 \qquad b_1\geq^1 s_1 >^1 a_1,
  \end{equation}
  which, together with  $c\y b < c\cdot b$ imply $s_2 = b_2$.
  We conclude that $c\cdot b\geq \lb a_1 , b_2 \rb = a\y b$, and hence
  \[
  \lb a_1',a_2\rb  = c = (c\cdot b)\cdot a  \geq (a\y b)\cdot a = a\y b 
  = \lb a_1 , b_2  \rb,  
  \]
  which contradicts the fact that $a_1' <^1 a_1$.
\end{proof}
The hypotheses on $C_i$ are necessary: The square of the 2-element chain admits
exactly two RRB operations.
\begin{corollary}\label{cor:2D-lattice-unique}
  The only RRB structure admissible for the Cartesian square of the naturals
  with the reverse order is given by the infimum. \qed
\end{corollary}

%% file: associative_trees.tex
\section{Foliated trees}
Some of our more general tools for proving associativity of posets involve
trees. In this paper, a \emph{tree} is a poset $(T,\leq)$, with top element $1$,
such that for every $x\in T$, $x{\uparrow}:=\{y\in T:x\leq y\}$ is linearly
ordered; a \emph{forest} is a disjoint union of trees. 

\begin{prop}\label{afi:compatibles} Let $T$ be a forest and $x,y\in T$. If there is a $z$ such that $z\leq x,y$, then $x$ and $y$ are comparable.
\end{prop}

We will say that a tree has \emph{finite branches} if every chain in the tree is finite.
Moreover, given a natural number $n$, we will say that a tree with finite branches has \emph{height n} if every chain has at most $n$ elements and there is at least one chain with $n$ elements.
Finally, we call a tree $T$ \emph{foliated} if for every $x\in T$, there is a minimal element below $x$.
Note that a foliated tree might still have branches without a minimal element.
\begin{theorem}\label{teo:arboles}
The following are equivalent (in $\ZF$):
\begin{enumerate}
\item\label{item:fol_as} Every foliated tree is associative.
\item\label{item:fin_as} Every tree with finite branches is associative.
\item\label{item:h3_as} Every tree with height $3$ is associative.
\item\label{item:AC} The Axiom of Choice.
\end{enumerate}
\end{theorem}
\begin{proof}
  \ref{item:fol_as}$\ent$\ref{item:fin_as} and
  \ref{item:fin_as}$\ent$\ref{item:h3_as} are trivial. Let us prove
  \ref{item:h3_as}$\ent$\ref{item:AC}. We will
show, using \ref{item:h3_as}, that every non empty family $\mathcal{F}$ of mutually disjoint non empty
sets has a transversal. Let $\mathcal{F}$ be such a family. We now define a
  tree order over $T\defi\{\calF\}\cup\calF\cup \bigcup\calF$: $x< y$
  if and only if $x\neq y$, and $y=\calF$ or $x\in y$ (we are considering
  $\calF\cap\bigcup\calF=\emptyset$. If this was not the case, the
  order obtained would not be that of a tree with height $3$. This can
  be fixed by considering $\{\calF \}\times\{2 \}\cup\calF\times\{1
  \}\cup \union\calF\times\{0\}$ as the universe for $T$ and defining
  the order as: $x'=\langle x,n\rangle<y'=\langle y, m\rangle$ if and
  only if $y'=\langle\calF,2\rangle$ or $x\in y$ and $n<m$). Note that
  this is a tree with height $3$ and therefore it is associative. Fix an
  admissible RRB structure for $T$ and a minimal $m\in T$. Then
  $\{m\cdot B: B\in\calF \}$ is transversal for $\calF$ as for
  $B\in\calF$, $m\cdot B$ is a minimal element below $B$ by
  Corollary~\ref{cor:minimal}.
  
  Let's now see that \ref{item:AC}$\ent$\ref{item:fol_as}. To this end,
  we fix a foliated tree $T$ and we invoke the Axiom of
  Choice to define an admissible RRB structure for $T$.

  Let $M:=\{x\in T:\text{$x$ is minimal}\}$. Define a well order over
  $M$ of type $\kappa$ for $\kappa$ a suitable ordinal. We can now
  think of $M$ as $M=\{x_{\alpha}:\alpha<\kappa \}$. We now proceed to
  decompose $T$ into disjoint convex chains $\{C_\alpha:
    \alpha<\kappa \}$. We define the chains $C_\alpha$ recursively.
  First, we take $C_0:=x_0{\uparrow}$. Now for $\alpha<\kappa$, we
  define
  $C_{\alpha}=x_{\alpha}{\uparrow}\setminus\bigcup_{\beta<\alpha}C_{\beta}$.
  \begin{claim}
    $\bigcup_{\alpha<\kappa}C_{\alpha}=T$.
  \end{claim}
  \begin{proof}
    $M\subset \bigcup_{\alpha<\kappa}C_{\alpha}$ is trivial Let $y\in
    T\setminus M$ and $M_y:=M\cap {y\downarrow}$, which is non-empty
    as the tree is foliated. Now, let
    $\alpha_y:=\min\{\beta<\kappa:x_{\beta}\in M_y\}$. By the
    construction of $\{C_{\alpha}:\alpha<\kappa \}$ we have that for
    every $\beta<\alpha_y$, $y\notin C_\beta$. Also $y\in
    x_{\alpha_y}{\uparrow}$ by hypothesis.  Then $y\in
    x_{\alpha_y}{\uparrow}\setminus\{C_{\beta}:\beta<\alpha_y\}=C_{\alpha_y}$.
  \end{proof}
  Let's now define a function $F$ in the following manner: For $y\in
  M$, $F(y)=y$, and for $y\in T\setminus M$, $F(y)=x_{\alpha_y}$ with
  $\alpha_y$ defined as in the proof of the last claim.
  \begin{claim}\label{afi:arboles2}
    Let $y,z$ such $F(y)\leq z\leq y$, then $F(z)=F(y)$.
  \end{claim}
  \begin{proof}
	If $z\in M$, then we must have $F(z)=z=F(y)$ and the result
        holds trivially. Let's check that the claim holds for $z\in
        T\setminus M$.  Note that since $M_z\subset M_y$, we get
        $\alpha_y\leq\alpha_z$. As we also have $x_{\alpha_y}\in M_z$
        by hypothesis, we obtain that $\alpha_z\leq\alpha_y$. Then
        $F(z)=x_{\alpha_z}=x_{\alpha_y}=F(y)$.
  \end{proof}
  Let's now define an RRB structure for $T$:
  \[
    x\cdot y=\begin{cases} \min\{x,y\} & \text{if they are
      comparable}\\ F(y) &\text{otherwise}
    \end{cases}
  \]

  Let $x,y,z$ in $T$. Then
  \[
    (x\cdot y)\cdot z=\begin{cases} \min\{x,y,z\} & (1) \\ F(z) &
    (2)\vee (4) \\ F(y) & (3)
    \end{cases}
  \]
  Where
  \begin{enumerate}
  \item $x$ and $y$ comparable and $z$ and
    $\min\{x,y\}$ comparable; or equivalently
    $x,y,z$ mutually comparable (By
    Proposition~\ref{afi:compatibles}).
  \item $x$ and $y$ comparable and $z$ and
    $\min\{x,y\}$ incomparable. We consider the
    following subcases:
    \begin{enumerate}
    \item $x\leq y$, $z$ and $x$ incomparable.
    \item $y< x$, $z$ and $y$ incomparable.
    \end{enumerate}
  \item$y$ incomparable with $x$ and $F(y)\leq
    z$. We consider the following subcases:
    \begin{enumerate}
    \item $F(y)\leq z\leq y$, $x$ and $y$
      incomparable.
    \item $F(y)\leq y< z$, $x$ and $y$
      incomparable.
    \end{enumerate}
    This classification is exhaustive because by
    Proposition~\ref{afi:compatibles} $z$ and $y$ must be comparable.
  \item $y$ incomparable with $x$ and $F(y)$ incomparable with $z$. We
    consider the following subcases:
    \begin{enumerate}
    \item $z<y$.
    \item $z$ and $y$ incomparable.
    \end{enumerate}
    There are no more subcases because $F(y)$ is minimal
    below $y$ and $F(y)$ is incomparable with $z$.
  \end{enumerate}
  Let's now check the value of $x\cdot(y\cdot z)$.
  \begin{enumerate}
  \item $x\cdot(y\cdot z)=\min\{x,y,z\}=(x\cdot
    y)\cdot z$.
  \item
    \begin{enumerate}
    \item It cannot be $y\leq z$ (as this would
      imply $x\leq z$), then either $z$ and $y$
      are incomparable or $z<y$. In the first
      case $x\cdot(y\cdot z)=x\cdot
      F(z)=F(z)=x\cdot z=(x\cdot y)\cdot z$. In
      the second case $x\cdot(y\cdot z)=x\cdot
      z=(x\cdot y)\cdot z$.
    \item $y$ and $z$ are incomparable. Then
      $x\cdot(y\cdot z)=x\cdot F(z)=F(z)=y\cdot
      z=(x\cdot y)\cdot z$.
    \end{enumerate}
  \item
    \begin{enumerate}
    \item 
      In this case we know by Proposition~\ref{afi:compatibles} that $x$ and $z$
      must be incomparable as $x$ and $y$ are so: $x\leq z$ would
      imply $x\leq y$, as $z\leq x$, together with
      Proposition~\ref{afi:compatibles} would imply that $x$ and $y$ are
      comparable.  We also have that $F(y)=F(z)$ by
      Claim~\ref{afi:arboles2}. Then $x\cdot(y\cdot z)=x\cdot
      z=F(z)=F(y)=F(y)\cdot z=(x\cdot y)\cdot z$.
    \item  $x\cdot(y\cdot z)=x\cdot y=F(y)=F(y)\cdot z=(x\cdot y)\cdot z$.
    \end{enumerate}
  \item
    \begin{enumerate}
    \item In this case $x,z$ are
      incomparable. Then $x\cdot(y\cdot z)=x\cdot
      z=F(z)=(x\cdot y)\cdot z$.
    \item $x\cdot(y\cdot z)=x\cdot F(z)=F(z)=F(y)\cdot z=(x\cdot y)\cdot z$.
    \end{enumerate}
  \end{enumerate}
\end{proof}

\begin{example}
We now present an application of Theorem \ref{teo:arboles}. Let $T$ be the tree in (a) of Figure~\ref{fig:foliados}. In (b) of Figure~\ref{fig:foliados} we can see a decomposition of $T$ into disjoint convex chains. For the RRB operation induced by this decomposition, we have that $x\cdot y=b$ and $y\cdot x=a$. 
  \begin{figure}[h!]
\begin{center}
\begin{tabular}{c@{\hspace{4em}}c}
      \begin{tikzpicture}
        [>=latex, thick,
          nodo/.style={thick,minimum size=0cm,inner sep=0cm}]
        
        \node (1) at (1.5,1.5) [nodo] {$\bullet$};
        \node (x) at (0,1) [label=left:$x$]   [nodo] {$\bullet$};
        \node (y) at (1,1) [label=right:$y$]  [nodo] {$\bullet$};
        \node (z) at (2,1)  [nodo] {$\bullet$};
        \node (a) at (-1,0)  [nodo] {$\bullet$};
        \node (b) at (0.5,0) [nodo] {$\bullet$};
        \node (c) at (1,0)  [label=right:$b$] [nodo] {$\bullet$};
        \node (d)  at (2,0) [nodo] {$\bullet$};
        \node (e)  at (0,-1) [nodo] {$\bullet$};
        \node (f)  at (0.5,-1) [nodo] {$\bullet$};
        \node (g)  at (1.5,-1) [nodo] {$\bullet$};
        \node (h)  at (2,-1) [nodo] {$\bullet$};
        \node (i)  at (-1,-1) [label=left:$a$]  [nodo] {$\bullet$};
        \draw [-] (y) edge  (1);
        \draw [-] (x) edge  (1);
        \draw [-] (z) edge  (1);
        \draw [-] (y) edge  (c);
        \draw [-] (a) edge  (x);
        \draw [-] (x) edge  (b);
        \draw [-] (d) edge  (z);
        \draw [-] (e) edge  (b);
        \draw [-] (f) edge  (b);
        \draw [-] (g) edge  (d);
        \draw [-] (h) edge  (d);
        \draw [-] (i) edge  (a);
        
      \end{tikzpicture}

&
      \begin{tikzpicture}
        [>=latex, thick,
          nodo/.style={thick,minimum size=0cm,inner sep=0cm}]
        
        \node (1) at (1.5,1.5) [nodo] {$\bullet$};
        \node (x) at (0,1) [label=left:$x$]   [nodo] {$\bullet$};
        \node (y) at (1,1) [label=right:$y$]   [nodo] {$\star$};
        \node (z) at (2,1)   [nodo] {$\diamond$};
        \node (a) at (-1,0) [nodo] {$\bullet$};
        \node (b) at (0.5,0) [nodo] {$\circ$};
        \node (c) at (1,0) [label=right:$b$] [nodo] {$\star$};
        \node (d)  at (2,0)  [nodo] {$\diamond$};
        \node (e)  at (0,-1) [nodo] {$\circ$};
        \node (f)  at (0.5,-1) [nodo] {$\vartriangle$};
        \node (g)  at (1.5,-1) [nodo] {$\diamond$};
        \node (h)  at (2,-1) [nodo] {$\odot$};
        \node (i)  at (-1,-1) [label=left:$a$]  [nodo] {$\bullet$};
        \draw [-] (y) edge  (1);
        \draw [-] (x) edge  (1);
        \draw [-] (z) edge  (1);
        \draw [-] (y) edge  (c);
        \draw [-] (a) edge  (x);
        \draw [-] (x) edge  (b);
        \draw [-] (d) edge  (z);
        \draw [-] (e) edge  (b);
        \draw [-] (f) edge  (b);
        \draw [-] (g) edge  (d);
        \draw [-] (h) edge  (d);
        \draw [-] (i) edge  (a);
        
      \end{tikzpicture}
  \\
      (a) & (b)
    \end{tabular}
    \end{center}
    \caption{An application of Theorem \ref{teo:arboles}. }\label{fig:foliados}
  \end{figure}
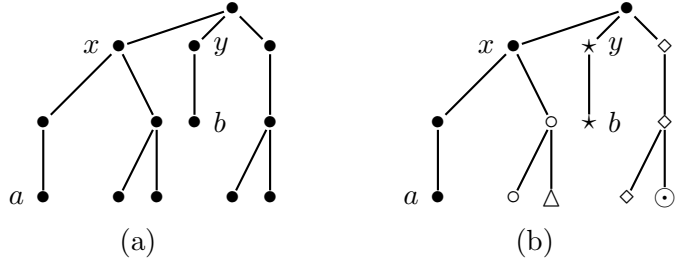
\end{example}

\begin{remark}
  If a tree has a minimal element but is not foliated, by Corollary~\ref{cor:minimal} it cannot be associative.
\end{remark}

\begin{remark}\label{obs:foresta}
  The previous theorem still holds if we chose $T$ to be a forest of
  foliated trees instead of just a foliated tree. This result follows
  from the fact that if we have a forest of foliated trees $T$, then we
  can consider the poset $T':=T+\{1\}$, where $1\notin T$. $T'$ is a foliated tree and is therefore
  associative by the previous theorem. Finally, since $T$ is a decreasing subset
  of $T'$, it
  is associative; in fact, it is a substructure with respect to the operation we
  defined in the previous theorem.
\end{remark}

The last theorem in this section will show that certain homomorphic preimages of
foliated trees are associative.

\begin{theorem}\label{th:homo} 
Let $P$ be a poset. Suppose there exist a forest $T$ consisting of foliated trees and a surjective homomorphism $f:P\rightarrow T$ such that: 
\begin{enumerate}
    \item 
$f(x)<f(y)$ implies $x<y$. 
\item
 For all $a\in f(P)$, $f^{-1}(a)$ is an associative subposet of $P$ and in addition, if $a$ is minimal in $T$, then $f^{-1}(a)$ has a minimum element.
\end{enumerate}
Then, $P$ is an associative poset.
 \end{theorem}
 \begin{proof}
We begin by establishing two claims which will be necessary for this proof: 
\begin{claim}\label{claim:homo}
If $f(z)\neq f(y)=f(x)$ and $z\leq y$, then $z\leq x$
\end{claim}
\begin{proof}
$z\leq y\implies f(z)\leq f(y)$, as $f(z)\neq f(y)$, $f(z)<f(y)=f(x)$, then $z<x$ by $1$.
\end{proof}
\begin{claim}\label{claim:homo2}
For all $x,y\in P$, if the set $\{x,y\}$ has a lower bound, then either $x$ and $y$ are comparable, or $f(x)=f(y)$.
\end{claim}
\begin{proof}
  Let $z$ be such a lower bound. That is $z\leq x$ and $z\leq y$. Then
  $f(z)\leq f(x)$ and $f(z)\leq f(y)$. As the codomain of $f$ is a
  forest, $f(x)$ and $f(y)$ are comparable by
  Proposition~\ref{afi:compatibles}. If $f(x)\neq f(y)$, then either $f(x)<f(y)$
  or $f(y)<f(x)$. In both of these cases $x$ and $y$ are comparable.
\end{proof}

Let's take an element $1$ not belonging to $T$. Let
$\tilde{F}:T+\{1\}\rightarrow T+\{1\}$ a function defined like the
one in Theorem~\ref{teo:arboles}. That is, a function such that $\tilde{F}(a)$
is minimal below $a$ for ever $a$ in $T$, and also, for every minimal
$b\in T$, $\tilde{F}^{-1}(b)$ is a convex chain containing $b$.

Let's define $F:P\rightarrow P$
as
\[
  F(x)=\min\{f^{-1}(\tilde{F}(f(x)))\}.
\]
Note that if $f(x)=f(y)$
then $F(x)=F(y)$.
\begin{claim}
  If $F(y)\leq z\leq y$ then $F(y)=F(z)$.
\end{claim}
\begin{proof}
  If $f(y)=f(z)$, the result holds trivially. Otherwise, we have that
  $f(z)<f(y)$. Note that $f(F(y))=\tilde{F}(f(y))\leq f(z)<f(y)$ and
  therefore by \ref{afi:arboles2},
  $\tilde{F}(f(z))=\tilde{F}(f(y))$. This tells us, by definition of
  $F$, that $F(z)=F(y)$.
\end{proof}
Let's take for every $a\in T$, an admissible RRB
$\tilde{\cdot}_a$ for $f^{-1}(a)$.  We now define a binary operation
$\cdot$ over $P$ in the following manner:
\begin{equation}
  \label{eq:prod-homo-onto-tree}
  x\cdot y=
  \begin{cases}
    \min\{x,y\} & x,y \text{ comparable}\\
    x\mathbin{\tilde{\cdot}_a}y  & f(x)=f(y)=a\\
    F(y) & x,y\text{ incomparable and }f(x)\neq f(y)
  \end{cases}
\end{equation}
Note that for all $x,y\in P$ we have $x\cdot F(y)=F(y)$: if they are
comparable, as $F(y)$ is minimal in $P$ we must have $F(y)\leq x$. If
$f(x)=f(F(y))=a$, then they are comparable (as $F(y)=\min
f^{-1}(f(F(y))$) and we can apply the previous reasoning. Otherwise,
we have that $x\cdot F(y)=F(F(y)))=F(y)$ by definition of $F$ and
$\tilde{F}$.

To lighten the notation we will omit the reference to $a$ in $\mathbin{\tilde{\cdot}_a}$, writing just $\tilde{\cdot}$ instead. The associative law $(x\cdot y)\cdot z=x\cdot(y\cdot z)$ for $x,y,z\in P$ is proved by the following case distinction. We first consider whether $f(x)$ is equal to $f(y)$. Next we consider whether $f(y)=f(z)$. Thirdly, we consider the comparability condition between $x$ and $y$. And lastly, we consider the comparability condition between $y$ and $z$. We  will showcase some representative cases.

Consider that $f(x)=f(y)\neq f(z)$ $x\leq y$ and $y,z$ incomparable. Then
$x\cdot(y\cdot z) = x\cdot F(z) = F(z) = x\cdot z = (x\cdot y)\cdot z$
because $x$ and $z$ are incomparable. This is because, if $x \leq y$ and $y$ is incomparable with $z$, it cannot be the case that $z \leq x$. It also cannot be $x \leq z$ because that would imply that ${y, z}$ has a lower bound, which contradicts \ref{claim:homo2}.
Now consider that $f(x)\neq f(y)=f(z)$ $x,y$ incomparable,$y\leq z$. Then
$(x\cdot y)\cdot z = F(y)\cdot z = F(y) = F(y\mathbin{\tilde{\cdot}} z) = x\cdot (y\mathbin{\tilde{\cdot}} z) = x\cdot (y\cdot z)$
because $F(y) \leq z$, $F(y) = F(y\mathbin{\tilde{\cdot}}z)$, and because $x$ is incomparable with $y\mathbin{\tilde{\cdot}} z$ (since it cannot be $x \leq y\mathbin{\tilde{\cdot}} z$ by ~\ref{claim:homo}, and it cannot be $y\mathbin{\tilde{\cdot}} z \leq x$ because that, together with ~\ref{claim:homo2}, would imply $z \leq x$ and that would mean $y \leq x$).
\end{proof}

\begin{example}
We now present an example of an application of Theorem \ref{th:homo}. The theorem tells us that the poset depicted in (a) of Figure \ref{fig:homo} is associative.
\begin{figure}[h]
\begin{center}
\begin{tabular}{c@{\hspace{4em}}c}

\begin{tikzpicture}
   [>=latex, thick,
          nodo/.style={thick,minimum size=0cm,inner sep=0cm}]
        \node (1) at (1.5,2) [nodo] {$\bullet$};
        \node (2) at (1,1.5) [nodo] {$\bullet$};
        \node (3) at (2,1.5) [nodo] {$\bullet$};
        \node (4) at (1,1) [nodo] {$\star$};
        \node (5) at (1,.5)  [nodo] {$\star$};
        \node (6) at (1,0) [nodo] {$\star$};
        \node (7) at (1.75,1) [nodo] {$\circ$};
         \node (8) at (2.25,1)  [nodo] {$\circ$};
        \node (9) at (1.75,.5) [nodo] {$\circ$};
        \node (10) at (2.25,.5) [nodo] {$\circ$};
        \node (11) at (2,0)  [nodo] {$\odot$};
       \node (12) at (0.5,-.5) [nodo] {$\diamond$};
       \node (13) at (0.5,-1)   [nodo] {$\diamond$};
      \node (14) at (1.25,-.5) [nodo] {$\vartriangle$} ;
     \node (15) at (1,-1) [nodo] {$\vartriangle$};
    \node (16) at (1.5,-1) [nodo] {$\vartriangle$};
    \node (17) at (1.25,-1.5) [nodo] {$\vartriangle$};
        \draw [-] (2) edge  (1);
        \draw [-] (3) edge  (1);
        \draw [-] (2) edge  (4);
        \draw [-] (2) edge  (7);
        \draw [-] (2) edge  (8);
        \draw [-] (3) edge  (4);
        \draw [-] (3) edge  (7);
        \draw [-] (3) edge  (8);
        \draw [-] (5) edge  (4);
        \draw [-] (5) edge  (6);
        \draw [-] (12) edge  (6);
        \draw [-] (6) edge  (14);
        \draw [-] (12) edge  (13);
        \draw [-] (15) edge  (14);
        \draw [-] (16) edge  (14);
        \draw [-] (15) edge  (17);
        \draw [-] (16) edge  (17);
        \draw [-] (9) edge  (7);
        \draw [-] (10) edge  (7);
        \draw [-] (9) edge  (8);
        \draw [-] (10) edge  (8);
        \draw [-] (11) edge  (9);
        \draw [-] (10) edge  (11);
        
\end{tikzpicture}
&
\begin{tikzpicture}
   [>=latex, thick,
          nodo/.style={thick,minimum size=0cm,inner sep=0cm}]
        \node (1) at (1.5,2) [nodo] {$\bullet$};
        \node (2) at (1,1.5)[nodo] {$\star$};
        \node (3) at (2,1.5)  [nodo] {$\circ$};
        \node (4) at (.5,1)   [nodo] {$\diamond$};
        \node (5) at (1.5,1) [nodo] {$\vartriangle$};
        \node (6) at (2,1) [nodo] {$\odot$};
  
        \draw [-] (2) edge  (1);
        \draw [-] (3) edge  (1);
        \draw [-] (2) edge  (4);
        \draw [-] (2) edge  (5);
        \draw [-] (3) edge  (6);
  
  \end{tikzpicture}
 \\
 (a) & (b)
\end{tabular}
\end{center}
\caption{ An application of Theorem \ref{th:homo}.}\label{fig:homo}
\end{figure}
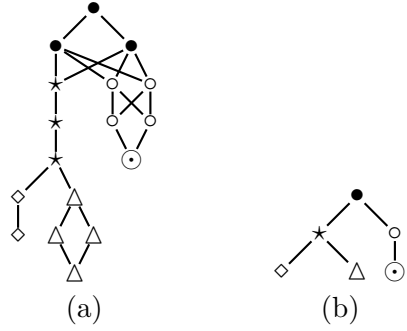
\end{example}

%% file: conclusion.tex
\section{Conclusion}
\label{sec:conclusion}

We have presented many examples which showcase the usefulness of an order
theoretical point of view for studying bands. We proved that we can define, for
certain posets, a band operation by invoking some of its structural properties,
such as the existence of a special order-preserving function for a normal poset,
or a decomposition into disjoint convex sub-chains with minimum for a foliated
tree. More so, we showed that in the first case, this is the \emph{only} way of
defining a right-normal band operation over a normal poset. Here, natural
definability questions arise. There appears to be no way of canonically
assigning a right-normal band operation to every normal poset. We can then ask
ourselves if being able to define this assignment is equivalent to the
Axiom of Choice or some fragment of it.

We proved the equivalence of the associativity of foliated trees and the Axiom
of Choice. This, together with the observations regarding right-normal bands,
shows that assigning a (right-regular or right-normal) band operation to an
associative poset is highly non-canonical. We believe that a broader family of
trees %
might be proven to be associative. It
appears to be that for associative trees there must exist, perhaps as some sort
of limit,
some order type that
``occurs densely''. That is, a type order which appears in arbitrary low levels of the
tree. Therefore, this leads us to believe that every tree admitting a
decomposition into \emph{disjoint convex sub-chains with isomorphic initial
segments} might be associative.

The following is a list of problems/questions which remain open.

 \begin{question} Is every associative disjoint union of meet-semilattices a normal poset?
\end{question}
We know that every disjoint union of meet-semilattices with minimum is normal, so if the answer to this question is negative, at least one of the meet-semilattices in our counterexample must be infinite.

Regarding the definability of RNB structures over normal posets, we have:

\begin{question} Is there a family of posets whose ``normality'' is equivalent to the Axiom of Choice?
\end{question}
A result of this sort would be an analogue of Theorem~\ref{teo:arboles}. We are
also interested in a weaker version of the previous question:
\begin{question}
Is the Axiom of Choice needed to define a right-normal band operation for every normal poset?
\end{question}

The following dwells on possible extensions of
Lemma~\ref{lem:union-disj} or Lemma~\ref{lem:sum-ord2}:
\begin{question}
Is every disjoint union of finite associative posets associative?
\end{question}
We believe the answer to this question to be negative. From Theorem~\ref{th:homo} we know that every disjoint union of associative posets with minimum is associative, so if a counterexample exists, at least one of the associative posets must not have a minimum.
\begin{question} Can Theorem~\ref{teo:arboles} be extended to prove associativity of a more general class of posets/trees?
\end{question}

\paragraph*{Acknowledgment}
We warmly thank the referee for their valuable feedback that greatly helped to
improve the paper, especially the pointer to Green's work and suggestions to our
naming conventions.

%% file: extra_material.tex
\section{Additional proofs}
\label{sec:additional-proofs}
\subsection{Multiplication tables}

We present here the multiplication tables of some of the examples presented. In particular, those shown in Section $3$. We begin giving the multiplication table for the RNB mentioned in Example~\ref{ex:norm}.

Consider the following labeling on the poset:

  \begin{figure}[h!]
    \begin{center}
        \begin{tikzpicture}
          [>=latex, thick,
            nodo/.style={thick,minimum size=0cm,inner sep=0cm}]
          
          \node (7) at (1.75,1) [label=left:$a$] [color=black!100, fill=red!0, nodo]  {$\bullet$};
          \node (8) at (2.25,1) [label=right:$b$] [color=black!100, fill=black!0, nodo]  {$\bullet$};
          \node (9) at (1.75,.5) [label=left:$c$] [color=black!100, fill=green!0, nodo] {$\bullet$};
          \node (10) at (2.25,.5) [label=right:$d$] [color=black!100, fill=cyan!0, nodo] {$\bullet$};
          \node (11) at (2,0) [label=below:$e$] [color=black!100, fill=blue!0, nodo] {$\bullet$};

        \draw [-] (9) edge  (7);
        \draw [-] (10) edge  (7);
        \draw [-] (9) edge  (8);
        \draw [-] (10) edge  (8);
        \draw [-] (11) edge  (9);
        \draw [-] (10) edge  (11);
		\end{tikzpicture}
		\end{center}
		\end{figure}

The multiplication table of the RNB obtained is: 

            \begin{center}
            \begin{tabular}{r | c c c c c}
                $\cdot$
                  & $a$ & $b$ & $c$ & $d$ & $e$ \\\hline
                $a$ & $a$ & $b$ & $c$ & $d$ & $e$ \\
                $b$ & $a$ & $b$ & $c$ & $d$ & $e$ \\
                $c$ & $c$ & $c$ & $c$ & $e$ & $e$ \\
                $d$ & $d$ & $d$ & $e$ & $d$ & $e$ \\
                $e$ & $e$ & $e$ & $e$ & $e$ & $e$\\
            \end{tabular}
		\end{center}

We will now present the multiplication tables for each of the operations obtained in  Example~\ref{ex:mult_norm}. Consider the following labeling on the poset:   
  \begin{figure}[h!]
    \begin{center}
        \begin{tikzpicture}
          [>=latex, thick,
            nodo/.style={thick,minimum size=0cm,inner sep=0cm}]
          
          \node (7) at (1.75,1) [label=left:$a$] [color=black!100, fill=red!0, nodo]  {$\bullet$};
          \node (8) at (2.25,1) [label=right:$b$] [color=black!100, fill=black!0, nodo]  {$\bullet$};
          \node (9) at (1.75,.5) [label=left:$c$] [color=black!100, fill=green!0, nodo] {$\bullet$};
          \node (10) at (2.25,.5) [label=right:$d$] [color=black!100, fill=cyan!0, nodo] {$\bullet$};
          \node (11) at (2,0) [label=below:$e$] [color=black!100, fill=blue!0, nodo] {$\bullet$};

          \draw [-] (9) edge  (7);
          \draw [-] (10) edge  (8);
          \draw [-] (11) edge  (9);
          \draw [-] (10) edge  (11);
		\end{tikzpicture}
		\end{center}
		\end{figure}

  The multiplication tables of the RNBs corresponding to the previous example are, respectively:
      \begin{table}[h]
        \begin{minipage}{0.5\textwidth}
            \begin{center}
            \begin{tabular}{r | c c c c c}
                  $\cdot$
                  & $a$ & $b$ & $c$ & $d$ & $e$ \\\hline
                $a$ & $a$ & $e$ & $c$ & $e$ & $e$ \\
                $b$ & $e$ & $b$ & $e$ & $d$ & $e$ \\
                $c$ & $c$ & $e$ & $c$ & $e$ & $e$ \\
                $d$ & $e$ & $d$ & $e$ & $d$ & $e$ \\
                $e$ & $e$ & $e$ & $e$ & $e$ & $e$\\
            \end{tabular}
            \end{center}
        \end{minipage}%
        \hspace{1em}
        \begin{minipage}{0.5\textwidth}
            \begin{center}
            \begin{tabular}{r | c c c c c}
 
      			$\cdot$
                  & $a$ & $b$ & $c$ & $d$ & $e$ \\\hline
                $a$ & $a$ & $d$ & $c$ & $d$ & $e$ \\
                $b$ & $c$ & $b$ & $c$ & $d$ & $e$ \\
                $c$ & $c$ & $d$ & $c$ & $d$ & $e$ \\
                $d$ & $c$ & $d$ & $c$ & $d$ & $e$ \\
                $e$ & $e$ & $e$ & $e$ & $e$ & $e$\\
            \end{tabular}
            \end{center}
        \end{minipage}%
                \hspace{1em}
            \begin{center}
            \begin{tabular}{r | c c c c c}
                $\cdot$
                  & $a$ & $b$ & $c$ & $d$ & $e$ \\\hline
                $a$ & $a$ & $b$ & $c$ & $d$ & $e$ \\
                $b$ & $a$ & $b$ & $c$ & $d$ & $e$ \\
                $c$ & $c$ & $d$ & $c$ & $d$ & $e$ \\
                $d$ & $c$ & $d$ & $c$ & $d$ & $e$ \\
                $e$ & $e$ & $e$ & $e$ & $e$ & $e$\\
            \end{tabular}
            \end{center}

    \end{table}
    
\subsection{Every poset having at most $4$ elements is associative}
Here we will show that every poset having at most $4$ elements is associative.
\begin{itemize}
\item Posets having at most $3$ elements are either a meet-semilattice or a forest of foliated trees (Remark \ref{obs:foresta}).
\item All but two posets with $4$ elements are either a meet-semilattice, a forest of foliated trees or an ordered sum of antichains (Lemma \ref{lem:sum-ord}).
\end{itemize}    
This leaves us to prove the associativity of only two $4$-element posets. These are:

 \begin{figure}[h!]
    \begin{center}
      \begin{tabular}{lcr}
        \begin{tikzpicture}
          [>=latex, thick,
            nodo/.style={thick,minimum size=0cm,inner sep=0cm}]
         
          \node (7) at (2.75,0) [label=right:$0$] [color=black!100, fill=red!0, nodo]  {$\bullet$};
          \node (9) at (1.75,.5) [label=left:$1$] [color=black!100, fill=green!0, nodo] {$\bullet$};
          \node (10) at (2.25,.5) [label=right:$2$] [color=black!100, fill=cyan!0, nodo] {$\bullet$};
          \node (11) at (2,0) [label=left:$3$] [color=black!100, fill=blue!0, nodo] {$\bullet$};

          \draw [-] (11) edge  (9);
          \draw [-] (10) edge  (11);
          
        \end{tikzpicture} & 
        \hspace{2em} & \begin{tikzpicture}
          [>=latex, thick,
            nodo/.style={thick,minimum size=0cm,inner sep=0cm}]
          \node (7) at (1.75,1) [label=left:$a$] [color=black!100, fill=red!0, nodo]  {$\bullet$};
          \node (9) at (1.75,.5) [label=left:$c$] [color=black!100, fill=green!0, nodo] {$\bullet$};
          \node (10) at (2.25,.5) [label=right:$d$] [color=black!100, fill=cyan!0, nodo] {$\bullet$};
          \node (11) at (2.25,1) [label=right:$b$] [color=black!100, fill=blue!0, nodo] {$\bullet$};
          
		  \draw [-] (7) edge  (9);
          \draw [-] (9) edge  (11);
          \draw [-] (10) edge  (11);
          
        \end{tikzpicture}
        \end{tabular}
        \end{center}
        \end{figure}
        For the first poset, note that we can prove its associativity using Theorem~\ref{th:homo}. Using this, we can provide a RRB structure given by: 
        
            \begin{center}
            \begin{tabular}{r | c c c c}
                $\cdot$
                  & $0$ & $1$ & $2$ & $3$ \\\hline
                $0$ & $0$ & $3$ & $3$ & $3$  \\
                $1$ & $0$ & $1$ & $3$ & $3$  \\
                $2$ & $0$ & $3$ & $2$ & $3$  \\
                $3$ & $0$ & $3$ & $3$ & $3$  \\ \\
            \end{tabular}
           \end{center}
           
           As for the second poset, the only admissible RRB is given by:
            \begin{center}
            \begin{tabular}{r | c c c c}
                $\cdot$
                  & $a$ & $b$ & $c$ & $d$ \\\hline
                $a$ & $a$ & $c$ & $c$ & $d$  \\
                $b$ & $c$ & $b$ & $c$ & $d$  \\
                $c$ & $c$ & $c$ & $c$ & $d$  \\
                $d$ & $c$ & $d$ & $c$ & $d$  \\ \\
            \end{tabular}
            \end{center}
\subsection{The crown is not associative}

In the next lemmas, we assume that there is an admissible RRB
structure for the crown poset. By minimality we obtain:
\begin{lemma}\label{lem:5y4por}
  \begin{enumerate}
  \item  $3 · 2 \in \{4, 5\}$.
  \item $4 · 1 \in\{ 3, 5\}$.
  \item $5 · 0 \in\{ 3, 4\}$.
  \end{enumerate}
\end{lemma}
So far, we have the situation pictured in Table~\ref{tb:crow-pogroupoid}.

\begin{lemma}\label{lem:equivalences}
  \begin{enumerate}
  \item $5 · 0 = 3$ if and only if $3 · 2 = 5$.
  \item $3 · 2 = 4$ if and only if $4 · 1 = 3$.
  \item $5 · 0 = 4$ if and only if $4 · 1 = 5$.
  \end{enumerate}
\end{lemma}
\begin{proof}
  For the first item, assume $5· 0 = 3$.   If $3 · 2 \neq 5$, by
  Lemma~\ref{lem:5y4por},   $3·2 =4$. We obtain, using
  Lemma~\ref{lem:aba-ba}(\ref{item:aba-ba}),
  \[
  4 = 4 · 0 = 3 · 2 · 0 = 5 · \underline{0 · 2 · 0} = 5 · 2 · 0 = 5 · 0 = 3,
  \]
  a contradiction. Thus we have the direct implication.
  For the converse, the map  given by the permutation $(53)(20)$ is an
  isomorphism.  

  For the other items, there are  isomorphisms that send the first equivalence
  to the other two.
\end{proof}

\begin{table}[h]
  \begin{center}
    \begin{tabular}{c|cccccc}
      $·$ & $0$ & $1$ & $2$ & $3$ & $4$ & $5$ \\\hline
      $0$  & $0$ & ? & ?       & $3$ & $4$ & $5$ \\
      $1$  & ? & $1$ & ?       & $3$ & $4$ & $5$ \\
      $2$  & ? & ? & $2$       & $3$ & $4$ & $5$ \\
      $3$  & $3$ & $3$ & $A$ & $3$ & $4$ & $5$ \\
      $4$  & $4$ & $B$ & $4$ & $3$ & $4$ & $5$ \\
      $5$  & $C$ & $5$ & $5$ & $3$ & $4$ & $5$ 
    \end{tabular}
  \end{center}
  \caption{The partial crown product; $A\in\{4,5\}$, $B\in\{3,5\}$, $C\in\{3,4\}$.}\label{tb:crow-pogroupoid}
\end{table}

\begin{prop}
  The crown poset does not admit an RRB structure.
\end{prop}
\begin{proof}
  Assume it is related to a right posemigroup. Then all previous lemmas apply.
  \begin{align*}
    4 · 1 = 3 &\iff 3 · 2 = 4  && \text{Lemma~\ref{lem:equivalences}}\\
    &\iff 3 · 2 \neq 5    && \text{Lemma~\ref{lem:5y4por}}\\
    &\iff 5 · 0 \neq 3   && \text{Lemma~\ref{lem:equivalences}}\\
    &\iff 5 · 0 = 4     && \text{Lemma~\ref{lem:5y4por}}\\
    &\iff 4 · 1 = 5   && \text{Lemma~\ref{lem:equivalences}.}
  \end{align*}
  This contradiction shows that Table~\ref{tb:crow-pogroupoid} can't
  be completed to obtain an associative product.
\end{proof}

\subsection{Preimages of foliated trees}
We present here the complete analysis of the case distinction required by the last paragraphs
of the proof of
Theorem~\ref{th:homo}.

Each case will be named by a $4$-tuple.  Its
first and second coordinates determine if
$f(y)$ is equal to $f(x)$ and $f(z)$ respectively; if $f(x)=f(y)$, then the
first number in the name will be $1$, otherwise it will be $2$. Similar remarks
hold for the second coordinate. The third and fourth coordinate determine the
comparability conditions between $x$ and $y$, and $y$ and $z$ respectively; if
$x\leq y$ the third coordinate will be $1$, if $y<x$ the third coordinate will
be $2$, and if $x$ and $y$ are incomparable, the third coordinate will be
$3$. The fourth coordinate behaves analogously. If an asterisk is present in
any coordinate, it means the reasoning applied in that case holds for all
possible choices for that coordinate.
As an example of these conventions, the name \textbf{[1.2.3.2]}
corresponds to the case in which $f(x)=f(y)$, $f(y)\neq f(z)$, $x$ and $y$ are
incomparable, and $z<y$.

In the following calculations, we denote the product defined in
(\ref{eq:prod-homo-onto-tree}) by juxtaposition.
\begin{description}
\item[{[1.1.$*$.$*$]}]  In this case:
\[(xy)z=(x\mathbin{\tilde{\cdot}} y)\mathbin{\tilde{\cdot}} z= x\mathbin{\tilde{\cdot}}( y\mathbin{\tilde{\cdot}} z)=x(yz)\]
\item[{[$*$.$*$.1.1]}] \[(xy)z=xz=x=xy=x(yz)\]
\item[{[$*$.$*$.1.2]}] \[(xy)z=xz=x(yz)\]
\item[{[$*$.$*$.2.1]}] \[(xy)z=yz=y=xy=x(yz)\]
\item[{[$*$.$*$.2.2]}] \[(xy)z=yz=z=xz=x(yz)\]

 \item[{[1.2.1.3]}] 
            \[x(yz)=xF(z)=F(z)=xz=(xy)z\] 
      
   \item[{[1.2.2.3]}]
            \[(xy)z=yz=F(z)=xF(z)=x(yz)\] 
            As $x\cdot F(z)=F(z)$.
       
    \item[{[1.2.3.1]}]
            \[(xy)z=(x \mathbin{\tilde{\cdot}} y)z=x \mathbin{\tilde{\cdot}} y=xy=x(yz)\]
            As $(x \mathbin{\tilde{\cdot}} y)<z$ by Claim ~\ref{claim:homo}.
        \item[{[1.2.3.2]}]
        \[(xy)z=(x \mathbin{\tilde{\cdot}} y)z=z=xz=x(yz)\] 
       by Claim ~\ref{claim:homo}.
        \item[{[1.2.3.3]}]
        \[(xy)z=(x \mathbin{\tilde{\cdot}} y)z=F(z)=xF(z)=x(yz)\] 
        because $(x\mathbin{\tilde{\cdot}} y)$ must be incomparable with $z$ by Claim~\ref{claim:homo2} as $x\mathbin{\tilde{\cdot}}y\leq y$.

         \item[{[2.1.1.3]}]
                \[(xy)z=xz=x=x(y\mathbin{\tilde{\cdot}}z)=x(yz)\] by Claim~\ref{claim:homo}.
        \item[{[2.1.2.3]}] 
          \[(xy)z=yz=y\mathbin{\tilde{\cdot}} z=x(y\mathbin{\tilde{\cdot}} z)=x(yz)\] 
  by  Claim~\ref{claim:homo}.
                \item[{[2.1.3.1]}] 
               \[(xy)z=F(y)z=F(y)=xy=x(yz)\]
               as $F(y)\leq y\leq z$.
                \item[{[2.1.3.2]}] 
                \[(xy)z=F(y)z=F(z)z=F(z)=xz=x(yz)\]
                because $F(y)=F(z)$ and the fact that $x$ and $z$ are incomparable .
                \item[{[2.1.3.3]}] 
                \[(xy)z=F(y)z=F(y)=F(y\mathbin{\tilde{\cdot}} z)=x(y\mathbin{\tilde{\cdot}} z)=x(yz)\]
   as $F(y)\leq z$, $F(y)=F(y\mathbin{\tilde{\cdot}}z)$ and $x$ and $y\mathbin{\tilde{\cdot}} z$ are incomparable.
    \item[{[2.2.1.3]}]
    \[(xy)z=xz=F(z)=xF(z)=x(yz)\]
    as $x$ and $z$ are incomparable.
                \item[{[2.2.2.3]}]
                \[(xy)z=yz=F(z)=xF(z)=x(yz)\]
                \item[{[2.2.3.1]}]
                \[(xy)z=F(y)z=F(y)=xy=x(yz)\]
                \item[{[2.2.3.2]}]
                \[(xy)z=F(y)z=F(z)=xz=x(yz)\]
                Because if $F(y)\leq z$ then $F(y)=F(z)$. Otherwise they are incomparable.
                \item[{[2.2.3.3]}]
                \[(xy)z=F(y)z=F(z)=xF(z)=x(yz)\] 
               because $F(y)$ and $z$ are incomparable.
         
\end{description}

%% file: associative_posets.bbl
\providecommand{\noopsort}[1]{}
\begin{small}\end{small}